\documentclass{article}
\usepackage[a4paper]{geometry}
\usepackage{latexsym}
\usepackage{enumerate}
\usepackage{amsthm,amsmath,amsfonts,amssymb}
\usepackage{graphicx,graphics}
\usepackage{xcolor}
\usepackage{epsfig,bm,epsf,float}
\usepackage{url}
\usepackage{hyperref}

\usepackage{minitoc}

\newcommand{\lcm}{{\textup{lcm}}}

\theoremstyle{plain}
\newtheorem{thm}{Theorem}[section]
\newtheorem{cor}[thm]{Corollary}
\newtheorem{lem}[thm]{Lemma}
\newtheorem{prop}[thm]{Proposition}
\newtheorem{property}[thm]{Property}

\newtheorem*{xthm}{Theorem}

\theoremstyle{definition}
\newtheorem{defn}[thm]{Definition}

\theoremstyle{remark}
\newtheorem{rem}[thm]{Remark}

\numberwithin{equation}{section}
\title{Superabundant numbers, their subsequences and the Riemann hypothesis}
\author{Sadegh Nazardonyavi, Semyon Yakubovich\\
\emph{Departamento de Matem\'{a}tica, Faculdade de Ci\^{e}ncias},\\ \emph{Universidade do Porto, 4169-007 Porto, Portugal}}
\begin{document}
\date{}
\maketitle
\begin{center}
Abstract
\end{center}

Let $\sigma(n)$ be the sum of divisors of a positive integer $n$. Robin's theorem states that the Riemann hypothesis is equivalent to the inequality $\sigma(n)<e^\gamma n\log\log n$ for all $n>5040$ ($\gamma$ is Euler's constant). It is a natural question in this direction to find a first integer, if exists, which violates this inequality. Following this process, we introduce a new sequence of numbers and call it as extremely abundant numbers. Then we show that the Riemann hypothesis is true, if and only if, there are infinitely many of these numbers. Moreover, we investigate some of their properties together with superabundant and colossally abundant numbers.


\section{Introduction}
There are several equivalent statements to the famous Riemann hypothesis(\S Introduction). Some of them are related to the asymptotic behavior of   arithmetic functions. In particular, the known Robin's criterion (theorem, inequality, etc.) deals with the upper bound of $\sigma(n)$. Namely,
\begin{xthm}\textup{\textbf{(Robin)}}
The Riemann hypothesis is true, if and only if,
\begin{equation}\label{Ro}
\forall n\geq5041,\qquad    \frac{\sigma(n)}{n\log\log n}<e^\gamma,
\end{equation}
where $\sigma(n)=\displaystyle{\sum_{d|n}d}$ and $\gamma$ is Euler's constant (\cite{Rob}, Th. 1).
\end{xthm}
Throughout this paper, as Robin used in \cite{Rob}, we let
\begin{equation}\label{f(n)}
f(n)=\frac{\sigma(n)}{n\log\log n}.
\end{equation}

In 1913, Gronwall \cite{GR} in his study of asymptotic maximal size for the sum of divisors of $n$, found that the order of $\sigma(n)$ is always "very nearly $n$" (\cite{HW}, Th. 323), proving
\begin{xthm}\textup{\textbf{(Gronwall)}}
\begin{equation}\label{GR}
    \limsup_{n\rightarrow\infty}f(n)=e^\gamma.
\end{equation}
\end{xthm}
Ramanujan in his unpublished manuscript \cite{Ram.N} proved that if $N$ is a generalized superior highly composite number, i.e. a number of CA which we introduce in the next section, then under the Riemann hypothesis
\begin{align*}
\liminf_{N\rightarrow\infty}\left(\frac{\sigma(N)}{N}-e^{\gamma}\log\log N\right)\sqrt{\log N}\geq& -e^\gamma(2\sqrt{2}+\gamma-\log4\pi)\approx-1.558,\\[5pt]
\text{and}&\\[5pt]
\limsup_{N\rightarrow\infty}\left(\frac{\sigma(N)}{N}-e^{\gamma}\log\log N\right)\sqrt{\log N}\leq& -e^\gamma(2\sqrt{2}-4-\gamma+\log4\pi)\approx-1.393.
\end{align*}
Later in 1984, Robin \cite{Rob} demonstrated that
\begin{equation}\label{rob2}
   f(n)\leq e^\gamma+\frac{0.6482\ldots}{(\log\log n)^2}\ ,\qquad (n\geq3),
\end{equation}
where $0.6482\ldots\approx(\frac{7}{3}-e^\gamma\log\log12)\log\log12$ and the left hand side of (\ref{rob2}) attains its maximum at $n=12$. In the same spirit, Lagarias \cite{Lag} proved that the Riemann hypothesis is equivalent to
$$
\sigma(n)\leq e^{H_n}\log H_n+H_n,\qquad(n\geq1),
$$
where $H_n=\sum_{j=1}^n 1/j$ and it is called the $n$-th harmonic number. Investigating upper and lower bounds of arithmetic functions, Landau (\cite{Land}, pp. 216-219) obtained the following limits:
$$
\liminf_{n\rightarrow\infty}\frac{\varphi(n)\log\log n}{n}=e^{-\gamma}\ ,\qquad \limsup_{n\rightarrow\infty}\frac{\varphi(n)}{n}=1,
$$
where $\varphi(n)$ is the Euler totient function, which is defined as the number of positive integers not exceeding $n$ that are relatively prime to $n$, and can also be expressed as a product extended over the distinct prime divisors of $n$ (see \cite[Th. 2.4]{Apostol}); i.e.
$$
\varphi(n)=n\prod_{p\mid n}\left(1-\frac1p\right).
$$
Furthermore, Nicolas \cite{Nic} proved that, if the Riemann hypothesis is true, then we have for all $k\geq2$,
\begin{equation}\label{nicolas ineq}
    \frac{N_k}{\varphi(N_k)\log\log N_k}>e^\gamma,
\end{equation}
where $N_k=\prod_{j=1}^k p_j$ and $p_j$ is the $j$-th prime. On the other hand, if the Riemann hypothesis is false, then there are infinitely many $k$ for which (\ref{nicolas ineq}) is true, and infinitely many $k$ for which (\ref{nicolas ineq}) is false.\\

Compared to numbers $N_k$ which are the smallest integers that maximize $n/\varphi(n)$, there are integers which play this role for $\sigma(n)/n$ and they are called superabundant numbers. In other words, $n$ is a superabundant number   (\cite{Al}, see also \cite{Ram.N}) if
\begin{equation*}\label{sa}
    \frac{\sigma(n)}{n}>\frac{\sigma(m)}{m}\qquad\mbox{for all}\  m<n.
\end{equation*}
Briggs \cite{Briggs} describes a computational study of the successive maxima of the relative sum-of-divisors function $\sigma(n)/n$. He also studies the density of these numbers. W\'{o}jtowicz \cite{Wojtowicz} showed that the values of $f$ are close to 0 on a set of asymptotic density 1. Another study on Robin's inequality can be found in \cite{Choie}; it was shown that Riemann hypothesis(RH) holds true, if and only if, every natural number divisible by a fifth power greater than 1 satisfies Robin's inequality.

In 2009, Akbary and Friggstad \cite{Ak} established the following interesting theorem which enables us to limit our attention to a narrow sequence of positive integers to find a probable counterexample to (\ref{Ro}).
\begin{xthm}[\cite{Ak}, Th. 3]
If there is any counterexample to Robin's inequality, then the least such counterexample is a superabundant number.
\end{xthm}

Unfortunately, to our knowledge, there is no known algorithm to compute superabundant number. Also (see \cite{Al}) the number of superabundants not exceeding $x$  is more than
$$
c\,\frac{\log x\log\log x}{(\log\log\log x)^2},
$$
and in \cite{Erd} it was even proved that for every $\delta<5/48$ this number exceeds
$$
(\log x)^{1+\delta},\qquad(x>x_0).
$$
As a natural question in this direction, it is interesting to determine the least number, if exists, that violates inequality (\ref{Ro}) which belongs to a thinner sequence of positive integers, and study its properties. Following this process, we introduce a new sequence of numbers and call its elements as \emph{extremely abundant numbers}. We will present in the sequel some of their properties. Surprisingly enough, it will be proved that the least number, if any, should be an extremely abundant number. Therefore, we will establish another criterion, which is equivalent to the Riemann hypothesis.

Before starting the main definition and results we mention the recent paper by Caveney et al. \cite{Caveney}. They defined a positive integer $n$ as an extraordinary number, if $n$ is composite and $f(n)\geq f(kn)$ for all
$$
k\in \mathbb{N}\cup\{1/p:\ p\ \text{is a prime factor of}\ n\}.
$$
Under these conditions, they showed that the smallest extraordinary number is $n = 4$. Then they proved that the Riemann hypothesis is true, if and only if, 4 is the only extraordinary number. For more properties of these numbers and comparison them with superabundant and colossally abundant numbers we refer the reader to \cite{Caveney-2}.
\section{Extremely abundant numbers}
\begin{defn}\label{extn}
A positive integer $n$ is an extremely abundant number, if either $n=10080$ or $n>10080$ and
\begin{equation*}\label{extn1}
\forall m\quad\text{\upshape{s.t.}} \quad10080\leq m<n,\qquad    \frac{\sigma(m)}{m\log\log m}<\frac{\sigma(n)}{n\log\log n}.
\end{equation*}
\end{defn}
Here 10080 has been chosen as the smallest superabundant number greater than 5040. In Table \ref{tabl} we list the first 20 extremely abundant numbers. To find them we used a list of superabundant numbers provided in \cite{Kil} and \cite{Noe}.

First let us call a positive integer $n$ (cf. \cite{Al} and \cite{Ram.N})
\begin{enumerate}[\upshape (i)]
    \item  \emph{colossally abundant}, if for some $\varepsilon>0$,
   \[\frac{\sigma(n)}{n^{1+\varepsilon}}>\frac{\sigma(m)}{m^{1+\varepsilon}},\qquad(m<n)\qquad\text{and}\qquad
    \frac{\sigma(n)}{n^{1+\varepsilon}}\geq\frac{\sigma(m)}{m^{1+\varepsilon}},\qquad (m>n);
    \]
  \item \emph{highly composite}, if $d(n)>d(m)$ for all $m<n$, where $d(n)=\sum_{d|n}1$;
  \item \emph{generalized superior highly composite}, if there is a positive number $\varepsilon$ such that
  $$
  \frac{\sigma_{-s}(n)}{n^\varepsilon}\geq\frac{\sigma_{-s}(m)}{m^\varepsilon},\qquad(m<n)\qquad\text{and}\qquad \frac{\sigma_{-s}(n)}{n^\varepsilon}>\frac{\sigma_{-s}(m)}{m^\varepsilon},  \qquad(m>n),
  $$
  where $\sigma_{-s}(n)=\sum_{d\mid n}d^{-s}$.
\end{enumerate}
The study of these classes of numbers was initiated by Ramanujan, in an unpublished part of his 1915 work on highly composite numbers (\cite{Ram.h}, \cite{Ram.N}, \cite{Ram.c}). More precisely, he defined rather general classes of these numbers. For instance, he defined generalized highly composite numbers, containing as a subset superabundant numbers (\cite{Ram.h}, section 59). Moreover, he introduced the generalized superior highly composite numbers, including as a particular case colossally abundant numbers. For more details about these numbers see \cite{Al}, \cite{Erd} and \cite{Ram.N}.\\

We denote the following sets of integers by
\begin{align*}
  SA =&\ \{n:\ \ n \mbox{\ is superabundant}\}, \\
  CA =&\ \{n:\ \ n \mbox{\ is colossally abundant}\}, \\
  XA =&\ \{n:\ \ n \mbox{\ is extremely abundant}\}.
\end{align*}
We also use SA, CA and XA as abbreviations of the corresponding sets. Clearly, $XA\neq CA$ (see Table \ref{tabl}). Indeed, we shall prove that infinitely many numbers of CA are not in XA and, that if RH holds, then infinitely many numbers of XA are in CA.

As an elementary result from the definition of extremely abundant numbers we have
\begin{prop}
The inclusion $XA\subset SA$ holds.
\end{prop}
\begin{proof}
First, $10080\in SA$. Further, if $n>10080$ and $n\in XA$, then, for $10080\leq m<n$, we have
\[
\frac{\sigma(n)}{n}=f(n)\log\log n>f(m)\log\log m=\frac{\sigma(m)}{m}.
\]\\
In particular, for $m=10080$, we get
$$
\frac{\sigma(n)}{n}>\frac{\sigma(10080)}{10080}.
$$
So that, for $m<10080$, we have
$$
\frac{\sigma(n)}{n}>\frac{\sigma(10080)}{10080}>\frac{\sigma(m)}{m},
$$
since $10080\in SA$. Therefore, $n$ belongs to SA.
\end{proof}

%
Next, motivating our construction of extremely abundant numbers, we will establish the first main result of the paper.
\begin{thm}\label{ext}
If there is any counterexample to Robin's inequality (\ref{Ro}), then the least one is an extremely abundant number.
\end{thm}
\begin{proof}
By doing some computer calculations we observe that there is no counterexample to Robin's inequality (\ref{Ro}) for $5040<n\leq10080$. Now let $n>10080$ be the least counterexample to inequality (\ref{Ro}). For $m$ satisfying $10080\leq m<n$ we have
$$
f(m)<e^\gamma\leq f(n).
$$
Therefore, $n$ is an extremely abundant.Ф
\end{proof}

As we mentioned in Introduction, we will prove an equivalent criterion to the Riemann hypothesis for which the proof is based on Robin's inequality (\ref{Ro}) and Gronwall Theorem. Let $\#A$ denote the cardinal number of a set $A$. The second main result is
\begin{thm}\label{rh}
The Riemann hypothesis is true if and only if $\# XA=\infty$.
\end{thm}
\begin{proof}

\emph{Sufficiency.} Assume that RH is not true, then from Theorem \ref{ext}, $f(m)\geq e^\gamma$ for some  $m\geq10080$. From Gronwall's theorem, $M=\sup_{n\geq10080}f(n)$ is finite and there exists $n_0$ such that $f(n_0)=M\geq e^\gamma$ (if $M=e^\gamma$ then set $n_0=m$). An integer $n>n_0$ satisfies $f(n)\leq M=f(n_0)$ and $n$ can not be in XA so that $\#XA\leq n_0$.\\

\emph{Necessity.} \ On the other hand, if RH is true, then Robin's inequality (\ref{Ro}) is true. If $\# XA$ is finite, then there exists an $m$ such that for every $n>m$, $f(n)\leq f(m)$. Then
$$
\limsup_{n\rightarrow\infty} f(n)\leq f(m)<e^\gamma,
$$
which is a contradiction to Gronwall's theorem.
\end{proof}
There are some primes, which cannot be the largest prime factors of any extremely abundant number. For example, referring to Table \ref{tabl}, there is no extremely abundant number with the largest prime factor $p(n)=149$. This can be deduced from Theorem \ref{p<log n prop}.

\section{Auxiliary Lemmas}
Before we state several properties of SA, CA and XA numbers, we give the following lemmas which will be needed in the sequel.
\begin{lem}\label{x(1+E)<y<x(1+E) inverse}
Let $a, b$ be positive constants and $x, y$ positive variables for which
$$
\log x>a,
$$
and
\begin{equation*}
x\left(1-\frac{a}{\log x}\right)<y<x\left(1+\frac{b}{\log x}\right),
\end{equation*}
Then
$$
y\left(1-\frac{c}{\log y}\right)<x<y\left(1+\frac{d}{\log y}\right),
$$
where
$$
c\geq b\left(1-\frac{b-\frac{b}{\log x}}{\log x+b}\right),\qquad d\geq a\left(1+\frac{a+\frac{b}{\log x}}{\log x-a}\right).
$$
\end{lem}
\begin{proof}
Dividing by $x$, inverting both sides and multiplying by $y$, we get
$$
\frac{y}{1+b/\log x}<x<\frac{y}{1-a/\log x}.
$$
We are looking for constants $c$ and $d$ such that
$$
1-\frac{c}{\log y}<\frac{1}{1+b/\log x},
$$
or equivalently
$$
c>(\log y)\frac{b}{\log x+b},
$$
and
$$
\frac{1}{1-a/\log x}<1+\frac{d}{\log y},
$$
or equivalently
$$
d>(\log y) \frac{a}{\log x-a}.
$$
First we determine $c$. Since
$$
y<x\left(1+\frac{b}{\log x}\right),
$$
then
$$
\log y<\log x+\log\left(1+\frac{b}{\log x}\right)<\log x+\frac{b}{\log x}.
$$
So that if
\begin{align*}
c>&\left(\log x+\frac{b}{\log x}\right)\frac{b}{\log x+b}\\
=&b\left(\log x+b-b+\frac{b}{\log x}\right)\frac{1}{\log x+b}\\
=&b\left(1-\frac{b-\frac{b}{\log x}}{\log x+b}\right),
\end{align*}
then
$$
c>\log y\frac{b}{\log x+b},
$$
and hence
$$
x>y\left(1-c\frac{\log\log y}{\log y}\right).
$$
Similarly, if
\begin{align*}
d>&\left(\log x+\frac{b}{\log x}\right)\frac{a}{\log x-a}\\
 =&\left(\log x-a+a+\frac{b}{\log x}\right)\frac{a}{\log x-a}\\
 =&a\left(1+\frac{a+\frac{b}{\log x}}{\log x-a}\right),
\end{align*}
then
$$
d>\log y \frac{a}{\log x-a},
$$
and therefore
$$
x<y\left(1+\frac{d}{\log y}\right).
$$
\end{proof}
By elementary differential calculus one can prove
\begin{lem}\label{x<y<cx hx}
Let $h(x)=\log\log x$. Then
$$
g(y)=\frac{y h(y)-xh(x)}{(y-x)h(x)},\qquad(y>x>e).
$$
is increasing. Especially, if $c>1$ and $e<x<y<c\,x$, we have $g(y)<g(c\,x)$.
\end{lem}
We will need in the sequel the following inequality
\begin{equation}\label{x<y<cx loglog x}
\frac{1}{c-1}\left(c\frac{\log\log cx}{\log\log x}-1\right)<1+\frac{c}{c-1}\frac{\log c}{\log x\log\log x}.
\end{equation}
Indeed, since
\begin{align*}
\frac{1}{c-1}\left(c\frac{\log\log cx}{\log\log x}-1\right)=&\frac{1}{c-1}\left(c\frac{\log\log cx-\log\log x}{\log\log x}+c-1\right)\\
=&\frac{1}{c-1}\left(\frac{c}{\log\log x}\log\left(1+\frac{\log c}{\log x}\right)+c-1\right)\\
<&\frac{1}{c-1}\left(\frac{c}{\log\log x}\left(\frac{\log c}{\log x}\right)+c-1\right)\\
=&1+\frac{c}{c-1}\frac{\log c}{\log x\log\log x}.
\end{align*}

\begin{lem}\label{sqrt>loglog}
Let $x\geq11$. For $y>x$ we have
$$
\frac{\log\log y}{\log\log x}<\frac{\sqrt{y}}{\sqrt{x}}.
$$
\end{lem}

Chebyshev's functions $\vartheta(x)$ and $\psi(x)$, which are respectively, the logarithm of the product of all primes $\leq x$, and the logarithm of the least common multiple of all integers $\leq x$; namely
$$
\vartheta(x)=\sum_{p\leq x}\log p,\qquad\psi(x)=\sum_{p^m\leq x}\log p=\sum_{p\leq x}\left\lfloor\frac{\log x}{\log p}\right\rfloor\log p,
$$
where $\lfloor x\rfloor$ is the largest integer not greater than $x$.
The prime number theorem is equivalent to (\cite{HW}, Th. 434; \cite{Ingham}, Th. 3, 12)
\begin{equation}\label{psi x-x}
    \psi(x)\sim x.
\end{equation}

\begin{lem}\label{ros-sch}
For $x>1$, we have
\begin{enumerate}[\upshape (i)]
  \item\hspace*{60pt} $\displaystyle{x\left(1-\frac{1.25}{\log x}\right)<\vartheta(x)<x\left(1+\frac{0.021}{\log x}\right)}$,
  \item\hspace*{60pt} $\displaystyle{x\left(1-\frac{0.9}{\log x}\right)<\psi(x)<x\left(1+\frac{0.5}{\log x}\right)}$.
\end{enumerate}
\end{lem}
\begin{proof}
\begin{enumerate}[\upshape (i)]
  \item We combined Corollary $2^\ast$ and Theorem $7^\ast$ ($5.6a^\ast$) of \cite{Sch-1976}.
  \item We appeal Corollary $2^\ast$ of \cite{Sch-1976} for $x\geq41$ and the fact $\psi(x)\geq\vartheta(x)$ and Computation for $x<41$. Using the inequality in Corollary of Theorem 14 \cite{Rosser-1962}; i.e.
      $$
      \psi(x)-\vartheta(x)<1.02\sqrt{x}+3\sqrt[3]{x}
      $$
      and (i)  we have
  \begin{align*}
   \psi(x)-x=&\psi(x)-\vartheta(x)+\vartheta(x)-x\\
            <&1.02\sqrt{x}+3\sqrt[3]{x}+(0.021)\frac{x}{\log x}\\
            =&\left(1.02\frac{\log x}{\sqrt{x}}+3\frac{\log x}{\sqrt[3]{x^2}}+0.021\right)\frac{x}{\log x}\\
            <&\frac12\frac{x}{\log x},\qquad(x\geq780).
  \end{align*}
          For $x<780$ we use computation.
\end{enumerate}
\end{proof}

\begin{lem}\label{ros-sch}

\begin{enumerate}[\upshape (i)]
   \item $$
\vartheta(x)>x\left(1-\frac{0.068}{\log x}\right),\qquad(x\geq89,909),
$$
  \item $$
\psi(x)<x\left(1+\frac{0.045}{\log x}\right),\qquad(x\geq43,730),
$$
  \item $$
\psi(x)>x\left(1-\frac{0.0221}{\log x}\right),\qquad(x\geq89,909).
$$
\end{enumerate}
\end{lem}
\begin{proof}
\begin{enumerate}[\upshape (i)]
  \item By Theorem $7^\ast$ ($5.5b^\ast$) of \cite{Sch-1976}, Theorems 2, 4 of \cite{Pereira}, for $89,909=x_0<x<10^{16}$ we have
  \begin{align*}
  \vartheta(x)-x=&\vartheta(x)-\psi(x)+\psi(x)-x\\
                >&-\sqrt{x}-\frac43\sqrt[3]{x}-0.0220646\frac{x}{\log x}\\
                =&-\left(\frac{\log x}{\sqrt{x}}+\frac43\frac{\log x}{\sqrt[3]{x^2}}+0.0220646\right)\frac{x}{\log x}\\
                \geq&-\left(\frac{\log x_0}{\sqrt{x_0}}+\frac43\frac{\log x_0}{\sqrt[3]{x_0^2}}+0.0220646\right)\frac{x}{\log x}\\
                >&-0.068\frac{x}{\log x},
  \end{align*}
  and by Theorem $7^\ast$ ($5.5b^\ast$) of \cite{Sch-1976}, Theorem 5 of \cite{Pereira} for $x\geq x_0=10^{16}$
\begin{align*}
  \vartheta(x)-x=&\vartheta(x)-\psi(x)+\psi(x)-x\\
                >&-1.001\sqrt{x}-1.1\sqrt[3]{x}-0.0220646\frac{x}{\log x}\\
                =&-\left(1.001\frac{\log x}{\sqrt{x}}+1.1\frac{\log x}{\sqrt[3]{x^2}}+0.0220646\right)\frac{x}{\log x}\\
                \geq&-\left(1.001\frac{\log x_0}{\sqrt{x_0}}+1.1\frac{\log x_0}{\sqrt[3]{x_0^2}}+0.0220646\right)\frac{x}{\log x}\\
                >&-0.068\frac{x}{\log x}.
\end{align*}
  \item By Theorem 3 of \cite{Pereira}, for $43,730=x_0\leq x\leq 10^8$
  \begin{align*}
    \psi(x)-x<&0.656\sqrt{x}+\frac43\sqrt[3]{x}\\
             =&\left(0.656\frac{\log x}{\sqrt{x}}+\frac43\frac{\log x}{\sqrt[3]{x^2}}\right)\frac{x}{\log x}\\
             \leq &\left(0.656\frac{\log x_0}{\sqrt{x_0}}+\frac43\frac{\log x_0}{\sqrt[3]{x_0^2}}\right)\frac{x}{\log x}\\
             <&0.045\frac{x}{\log x}
  \end{align*}
  By Theorem $7^\ast$ ($5.6a^\ast$) of \cite{Sch-1976} and Theorem 4 of \cite{Pereira}, for $10^8=x_0\leq x\leq 10^{16}$
  \begin{align*}
    \psi(x)-x=&\psi(x)-\vartheta(x)+\vartheta(x)-x\\
             <&\sqrt{x}+\frac65\sqrt[3]{x}+0.0201384\frac{x}{\log x}\\
             =&\left(\frac{\log x}{\sqrt{x}}+\frac65\frac{\log x}{\sqrt[3]{x^2}}+0.0201384\right)\frac{x}{\log x}\\
             \leq &\left(\frac{\log x_0}{\sqrt{x_0}}+\frac65\frac{\log x_0}{\sqrt[3]{x_0^2}}+0.0201384\right)\frac{x}{\log x}\\
             <&0.0229\frac{x}{\log x}
  \end{align*}
  Again by Theorem $7^\ast$ ($5.6a^\ast$) of \cite{Sch-1976} and Theorem 5 of \cite{Pereira}, for $x\geq x_0=10^{16}$
  \begin{align*}
    \psi(x)-x=&\psi(x)-\vartheta(x)+\vartheta(x)-x\\
             <&1.001\sqrt{x}+1.1\sqrt[3]{x}+0.0201384\frac{x}{\log x}\\
             =&\left(1.001\frac{\log x}{\sqrt{x}}+1.1\frac{\log x}{\sqrt[3]{x^2}}+0.0201384\right)\frac{x}{\log x}\\
             \leq &\left(1.001\frac{\log x_0}{\sqrt{x_0}}+1.1\frac{\log x_0}{\sqrt[3]{x_0^2}}+0.0201384\right)\frac{x}{\log x}\\
             <&0.020139\frac{x}{\log x}
  \end{align*}
  \item By Theorem $7^\ast$ ($5.5b^\ast$) of \cite{Sch-1976}.
\end{enumerate}
\end{proof}
\begin{lem}[\cite{Dusart-1}, \cite{Dusart-2}]\label{Dusart1,2}
We have
$$
\pi(x)>\frac{x}{\log x}\left(1+\frac{1}{\log x}\right),\qquad(x\geq599),
$$
$$
\pi(x)<\frac{x}{\log x}\left(1+\frac{1.2762}{\log x}\right),\qquad(x>1).
$$
where $\pi(x)$ is the number of primes $\leq x$.
\end{lem}

Littlewood oscillation for Chebyshev $\psi$ function is given by
\begin{lem}[\cite{Ingham}, Th. 34]\label{omega ingham}
We have
$$
\psi(x)-x=\Omega_\pm(x^{1/2}\log\log\log x),\qquad{\text{\upshape as}}\quad x\rightarrow\infty.
$$
More precisely,
$$
\varlimsup_{x\rightarrow\infty}\frac{\psi(x)-x}{x^{1/2}\log\log\log x}\geq\frac12,\qquad\varliminf_{x\rightarrow\infty}\frac{\psi(x)-x}{x^{1/2}\log\log\log x}\leq-\frac12.
$$
\end{lem}
\begin{lem}[\cite{Rosser-1962}, Th. 13]\label{psi-theta}
$$
\psi(x)-\vartheta(x)<1.42620\sqrt{x},\qquad(x>0)
$$
\end{lem}
Combining Lemmas \ref{omega ingham} and \ref{psi-theta}
\begin{cor}\label{thetax-x omega}
We have
$$
\vartheta(x)-x=\Omega_\pm(x^{1/2}\log\log\log x),\qquad{\text{\upshape as}}\quad x\rightarrow\infty.
$$
\end{cor}

\begin{lem}[\cite{Ingham-2}]\label{q, q+cqt}
The number of primes in the interval $(q, q+c\,q^\theta)$ is asymptotic to $c\,q^\theta/\log q$, where $\theta\geq5/8$.
\end{lem}

\begin{lem}[\cite{Dusart-2}, Prop. 1.10]\label{pk+1<pk(1+1/log2 pk)}
For $k\geq463$,
$$
p_{k+1}\leq p_k\left(1+\frac{1}{2\log^2 p_k}\right).
$$
\end{lem}

\begin{lem}[\cite{Dusart-3}, Th. 6.12]\label{prod1/(1-1/p)>e...}
We have
$$
\prod_{p\leq x}\left(\frac{p}{p-1}\right)<e^\gamma(\log x)\left(1+\frac{0.2}{\log^2x}\right),\qquad(x\geq2973),
$$
and
$$
\prod_{p\leq x}\left(\frac{p}{p-1}\right)>e^\gamma(\log x)\left(1-\frac{0.2}{\log^2x}\right),\qquad(x>1).
$$
\end{lem}

We will use the following inequality frequently
\begin{equation}\label{log(1+t)}
    \frac{t}{1+t}<\log(1+t)< t,\qquad (t>0),
\end{equation}
\section{Some properties of SA, CA and XA numbers}
This section is divided to three paragraphs, which we will exhibit several properties of superabundant, colossally abundant and extremely abundant numbers, respectively. When there is no ambiguity, we simply denote by $p$ the largest prime factor of $n$.
\begin{defn}
Let $g$ be a real-valued function and $A=\{a_n\}$ be an increasing sequence of integers. We say that $g$ is an increasing (decreasing) function on $A$ (or for $a_n\in A$), if $g(a_n)\leq g(a_{n+1})$ ($g(a_n)\geq g(a_{n+1})$) for all $n\in I$.
\end{defn}
\paragraph{\textbf{Superabundant Numbers}}
A positive integer $n$ is said to be superabundant if
$$
\frac{\sigma(n)}{n}>\frac{\sigma(m)}{m}\qquad\mbox{for all}\  m<n.
$$
In the starting point, we show that for any real positive $x\geq1$, there is at least one superabundant number in the interval $[x,2x)$. In other words
\begin{prop}\label{an+1<=2an}
Let $n<n'$ be two consecutive superabundant numbers. Then
$$
\frac{n'}{n}\leq2.
$$
\end{prop}
\begin{proof}
Let $n=2^{k_2}\cdots p$. We compare $n$ with $2n$. In fact
$$
\frac{\sigma(2n)/(2n)}{\sigma(n)/n}=\frac{2^{k_2+2}-1}{2^{k_2+2}-2}>1.
$$
Hence, $n'\leq2n$.
\end{proof}

\begin{prop}[\cite{Al}, Th. 2]\label{Th 2}
Let $n=2^{k_2}\cdots q^{k_q}\cdots r^{k_r}\cdots p$ be a superabundant number, $q<r$, and
$$
\beta:= \left\lfloor\frac{ k_q \log q}{\log r}\right\rfloor,
$$
where $\lfloor x\rfloor$ is the greatest integer less than or equal to $x$.
Then $k_r$ has one of the three values : $\beta-1$, $\beta+1$, $\beta$.
\end{prop}
In the next theorem we give a lower bound for the exponent $k_q$ related to the largest prime factor of $n$.

\begin{thm}\label{log p/log q<kq}
Let $n=2^{k_2}\cdots q^{k_q}\cdots p$ be a superabundant number with $2\leq q\leq p$. Then
$$
\left\lfloor\frac{\log p}{\log q}\right\rfloor\leq k_q.
$$
\end{thm}
\begin{proof}
Let $k_q=k$ and suppose that $k\leq[\log p/\log q]-1$. Hence
\begin{equation}\label{qk<p}
    q^{k+1}<p.
\end{equation}
Now we compare values of $\sigma(s)/s$, taking $s=n$ and $s=m=nq^{k+1}/p$. Since $\sigma(s)/s$ is multiplicative, we restrict our attention to different factors. But $n$ is superabundant and $m<n$. Thus
\[
1<\frac{\sigma(n)/n}{\sigma(m)/m}=\frac{q^{2k+2}-q^{k+1}}{q^{2k+2}-1}\left(1+\frac{1}{p}\right)=\frac{1}{1+1/q^{k+1}}\left(1+\frac{1}{p}\right).
\]
Consequently, $p<q^{k+1}$, which contradicts (\ref{qk<p}).
\end{proof}

\begin{prop}[\cite{Al}, Th. 5]\label{Th 5}
Let $n=2^{k_2}\cdots q^{k_q}\cdots p$ be a superabundant number. If $k_q=k$ and $q<(\log p)^{\alpha}$, where $\alpha$ is a constant, then
\begin{equation}\label{log q(k+1)}
\log\frac{q^{k+1}-1}{q^{k+1}-q}>\frac{\log q}{p\log p}\left\{1+O\left(\frac{(\log\log p)^2}{\log p\log q}\right)\right\},
\end{equation}
\begin{equation}\label{log q(k+2)}
\log\frac{q^{k+2}-1}{q^{k+2}-q}<\frac{\log q}{p\log p}\left\{1+O\left(\frac{(\log\log p)^2}{\log p\log q}\right)\right\}.
\end{equation}
\end{prop}
\begin{rem}[\cite{Al}, p. 453]\label{O(delta)}
Let $\delta$ denote the error term
\begin{align*}
\delta&=\frac{(\log\log p)^2}{\log p\log q},\qquad(q^{1-\theta}<\log p)\\
\delta&=\frac{\log p}{q^{1-\theta}\log q},\qquad(q^{1-\theta}>\log p)
\end{align*}
Then

\begin{equation}\label{log q(k+1)}
\log\frac{q^{k+1}-1}{q^{k+1}-q}>\frac{\log q}{\log p}\log\left(1+\frac1p\right)\left\{1+O(\delta)\right\},
\end{equation}
\begin{equation}\label{log q(k+2)}
\log\frac{q^{k+2}-1}{q^{k+2}-q}<\frac{\log q}{\log p}\log\left(1+\frac1p\right)\left\{1+O(\delta)\right\}.
\end{equation}

\end{rem}

\begin{cor}\label{cplogp<2k<c'plogp}
Let $n=2^{k_2}\cdots p$ be a superabundant number. Then there exist two positive constants $c$ and $c'$ such that
$$
cp\log p<2^{k_2}<c'p\log p.
$$
\end{cor}
\begin{proof}
By inequality (\ref{log(1+t)})
\begin{align*}
\log\frac{q^{k+1}-1}{q^{k+1}-q}=\log\left(1+\frac{q-1}{q^{k+1}-q}\right)<\frac{q-1}{q^{k+1}-q}\leq\frac{1}{q^k}
\end{align*}
and (\ref{log q(k+1)}), there exists a $C\,'>0$ such that
$$
q^k<C\,'\,\frac{p\log p}{\log q}.
$$
On the other hand, again from inequality (\ref{log(1+t)})
\begin{align*}
\log\frac{q^{k+2}-1}{q^{k+2}-q}=\log\left(1+\frac{q-1}{q^{k+2}-q}\right)>\frac{q-1}{q^{k+2}-1}>\frac{1}{2q^{k+1}}
\end{align*}
and (\ref{log q(k+2)}), there exists a $C>0$ such that
$$
q^k>C\,\frac{p\log p}{\log q}.
$$
Putting $c=C/\log q$,  $c\,'=C\,'/\log q$  we get the result.
\end{proof}
\begin{rem}\label{q k<2 k+2}
In \cite{Al} it was proved that $q^{k_q}<2^{k_2+2}$, and in p. 455, it was remarked that for large superabundant $n$, $q^{k_q}<2^{k_2}$ for $q>11$.
\end{rem}
\begin{cor}\label{p(n)<2 k-1 cor}
For large enough superabundant number $n=2^{k}\cdots p$
\begin{equation}\label{p<2 k-1}
p<2^{k-1}.
\end{equation}
\end{cor}
\begin{prop}
Let $n=2^{k}\cdots p$ be a superabundant number. Then for large enough $n$
$$
\left\lfloor\frac{k\log2}{\log p}\right\rfloor=1.
$$
\end{prop}
\begin{proof}
By Corollary \ref{cplogp<2k<c'plogp}
$$
\log(cp\log p)<k\log 2<\log(c'p\log p).
$$
Hence, for large enough $p$
$$
1<1+\frac{\log(c\log p)}{\log p}<\frac{k\log2}{\log p}<1+\frac{\log(c'\log p)}{\log p}<2.
$$
Therefore,
$$
\left\lfloor\frac{k\log2}{\log p}\right\rfloor=1.
$$
\end{proof}


\begin{prop}[\cite{Al}, Th. 7]\label{p=log}
If $n=2^k\cdots p$ is a superabundant number, then
$$
p\sim\log n.
$$
\end{prop}
From Corollary \ref{cplogp<2k<c'plogp} and Proposition \ref{p=log} it follows that
\begin{prop}
If $n=2^{k_2}\cdot3^{k_3}\cdots p(n)$ is a superabundant number, then for large enough $n$
\[
\log n<2^{k_2}.
\]
\end{prop}

\begin{proof}
We use Remark \ref{q k<2 k+2}, Lemma \ref{Dusart1,2} and Corollary  \ref{p(n)<2 k-1 cor} to get
\begin{align*}
  \frac{\log n}{2^{k_2}} &= \frac{\sum \log q^{k_q}}{2^{k_2}} \\
         &< \frac{5\log2^{k_2+2}+(\pi(p(n))-5)\log2^{k_2})}{2^{k_2}} \\
         &= \pi(p(n))\frac{\log2^{k_2}}{2^{k_2}}+\frac{10\log2}{2^{k_2}} \\
         &< \frac{p(n)}{\log p(n)}\left(1+\frac{1.2762}{\log p(n)}\right)\frac{\log 2^{k_2}}{2^{k_2}}+\frac{10\log2}{2^{k_2}}\\
         &= \frac{p(n)}{2^{k_2}}\frac{\log 2^{k_2}}{\log p(n)}\left(1+\frac{1.2762}{\log p(n)}\right)+\frac{10\log2}{2^{k_2}}\\
         &<1.
\end{align*}
\end{proof}
\begin{prop}
Let $n=2^{k_2}\cdots q^{k_q}\cdots p$ be a superabundant number. Then

\begin{equation}\label{psi(p(n))<log n}
\psi(p)\leq\log n.
\end{equation}
Moreover,
\begin{equation}\label{psilog}
    \lim_{n\rightarrow\infty}\frac{\psi(p)}{\log n}=1.
\end{equation}
\end{prop}
\begin{proof} In fact, by Theorem \ref{log p/log q<kq}
$$
\psi(p)=\sum_{q\leq p}\left\lfloor\frac{\log p}{\log q}\right\rfloor\log q\leq\sum_{q\leq p}k_q\log q=\log n.
$$
In order to prove (\ref{psilog}) we appeal to (\ref{psi x-x}) and Proposition \ref{p=log}.
\end{proof}

\begin{lem}
For large enough $n=2^k\cdots q^{k_q}\cdots p\in SA$
$$
\frac{\log n}{\vartheta(p)}<1+\frac{c\sqrt{2\log p}}{\sqrt{p}}\left(1+\frac{1.5}{\log p}\right)
$$
\end{lem}
\begin{proof}
Let $x_2$ be the largest prime factor with exponent 2. From the error term in page 453 of Alaoglu-Erd\H{o}s $x_2^2<2p\log p$ for large enough $n\in SA$.
\begin{align*}
\log n-\vartheta(p)=&\sum_{2\leq q\leq x_2}(k_q-1)\log q<c\,\vartheta(x_2)\\
\leq& c\,\vartheta(\sqrt{2p\log p})<c\sqrt{2p\log p}\left(1+\frac{2(0.021)}{\log2p\log p}\right)\\
\end{align*}
So
\begin{align*}
\frac{\log n}{\vartheta(p)}-1<\frac{c\sqrt{2p\log p}\left(1+\frac{2(0.021)}{\log2p\log p}\right)}{p\left(1-\frac{1.25}{\log p}\right)}<\frac{c\sqrt{2\log p}}{\sqrt{p}}\left(1+\frac{1.5}{\log p}\right),\qquad(p>5342)
\end{align*}

\end{proof}

\begin{prop}\label{p(1+e)<log an<p(1+e)}
For $n=2^{k}\cdots p\in SA$ we have
$$
\log n>p\left(1-\frac{0.0221}{\log p}\right),
$$
and for large enough $n\in SA$
$$
\log n<p\left(1+\frac{0.5}{\log p}\right).
$$
\end{prop}

\begin{proof}
The first inequality holds by (\ref{psi(p(n))<log n}) and Lemma \ref{ros-sch}. Concerning the second inequality \begin{align*}
\frac{\log n}{p}=&\frac{\log n}{\vartheta(p)}\,\frac{\vartheta(p)}{p}\\
<&\left\{1+\frac{c\sqrt{2\log p}}{\sqrt{p}}\left(1+\frac{1.5}{\log p}\right)\right\}\left(1+\frac{0.021}{\log p}\right)\\
<&\left(1+\frac{0.5}{\log p}\right)
\end{align*}

\end{proof}
From Lemma \ref{x(1+E)<y<x(1+E) inverse} and Proposition \ref{p(1+e)<log an<p(1+e)}, we conclude
\begin{cor}\label{log an<p<log an}
For large enough $n=2^{k}\cdots p\in SA$, it has
$$
\log n\left(1-\frac{0.5}{\log\log n}\right)<p<\log n\left(1+\frac{0.0222}{\log\log n}\right).
$$
\end{cor}
In Sections 18.3 and 18.4 of \cite{HW}, it was proved that
$$
\frac{6}{\pi^2}<\frac{\sigma(n)\varphi(n)}{n^2}<1,
$$
and
$$
\varliminf_{n\rightarrow\infty}\frac{\sigma(n)\varphi(n)}{n^2}=\frac{6}{\pi^2},\qquad\qquad\varlimsup_{n\rightarrow\infty}\frac{\sigma(n)\varphi(n)}{n^2}=1.
$$
\begin{prop}\label{sigma/n>(1-c)n/phi}
For $n=2^{k_2}\cdots q^{k_q}\cdots p\in SA$, we have
$$
\frac{\sigma(n)}{n}>\left\{1-\varepsilon(p)\right\}\frac{n}{\varphi(n)},
$$
where
$$
\varepsilon(p)=\frac{1}{\log p}\left(1+\frac{1.5}{\log p}\right).
$$
\end{prop}
\begin{proof}
It is enough to show that
$$
\prod_{q\leq p}\left(1-\frac{1}{q^{k_q+1}}\right)>1-\frac{1}{\log p}\left(1+\frac{1.5}{\log p}\right).
$$
Hence, using logarithmic inequality
$$
\log\left(1-\frac1t\right)>-\frac{1}{t-1},\qquad(t>1),
$$
and Theorem \ref{log p/log q<kq} and Lemma \ref{Dusart1,2}, we obtain
\begin{align*}
\log\prod_{q\leq p}\left(1-\frac{1}{q^{k_q+1}}\right)=&\sum_{q\leq p}\log\left(1-\frac{1}{q^{k_q+1}}\right)>-\sum_{q\leq p}\frac{1}{q^{k_q+1}-1}\\
>&-\sum_{q\leq p}\frac{1}{q^{{\log p/\log q}}-1}=-\sum_{q\leq p}\frac{1}{p-1}\\
=&-\frac{\pi(p)}{p-1}>-\frac{1}{p-1}\frac{p}{\log p}\left(1+\frac{1.2762}{\log p}\right)\\
>&-\frac{1}{\log p}\left(1+\frac{1.5}{\log p}\right),\qquad(p\geq23).
\end{align*}
Therefore, taking the exponential of both sides and using $e^{-x}>1-x$ we get
$$
\prod_{q\leq p}\left(1-\frac{1}{q^{k_q+1}}\right)>1-\frac{1}{\log p}\left(1+\frac{1.5}{\log p}\right),\qquad(p\geq23).
$$
For $p<23$, we use computation.
\end{proof}

\begin{prop}\label{|f(n)-ey|<1/loglogn2}
Let $n=2^{k}\cdots p\in SA$. Then
$$
\lim_{n\rightarrow\infty}\frac{\sigma(n)}{n\log\log n}=e^\gamma.
$$
More precisely,
$$
\frac{\sigma(n)}{n\log\log n}<e^\gamma\left(1+\frac{0.3639\ldots}{(\log\log n)^2}\right),\qquad(n\geq3)
$$
and for large enough $n\in SA$
\begin{equation}\label{sigma n/(nloglogn)>}
\frac{\sigma(n)}{n\log\log n}>e^\gamma\left(1-\frac{2}{\log\log n}\right).
\end{equation}

\end{prop}
\begin{proof}
The first inequality is exactly (\ref{rob2}), where $0.3639\ldots=(0.6482\ldots)e^{-\gamma}$.

By using Proposition \ref{sigma/n>(1-c)n/phi}, Lemma \ref{prod1/(1-1/p)>e...}  and Corollary \ref{log an<p<log an}, we have for large enough $n$
\begin{align*}
\frac{\sigma(n)}{n}>&\left\{1-\varepsilon(p)\right\}e^{\gamma}(\log p)\left(1-\frac{0.2}{\log^2p}\right)\\
                       =& e^{\gamma}(\log p)\left\{1-\frac{1}{\log p}\left(1+\frac{1.5}{\log p}\right)\right\}\left(1-\frac{0.2}{\log^2p}\right)\\
>& e^{\gamma}(\log\log n)\left(1-\frac{2}{\log\log n}\right).
\end{align*}

\end{proof}

\begin{lem}\label{sigma n-n}
Let
$$
g(n)=\sigma(n)-n.
$$
Then $g$ is increasing for $a_n\in SA$.
\end{lem}
\begin{proof}
Let $a_n, a_{n+1}\in SA$. By definition of superabundant numbers
$$
\frac{\sigma(a_{n+1})}{a_{n+1}}>\frac{\sigma(a_n)}{a_n}>1,\qquad(n>1).
$$
Therefore,
\begin{align*}
\frac{\sigma(a_{n+1})}{\sigma(a_n)}>\frac{a_{n+1}}{a_n}\Rightarrow&\frac{\sigma(a_{n+1})}{\sigma(a_n)}-1>\frac{a_{n+1}}{a_n}-1\\
\Rightarrow&\sigma(a_n)\left(\frac{\sigma(a_{n+1})}{\sigma(a_n)}-1\right)>a_n\left(\frac{a_{n+1}}{a_n}-1\right)\\
\Rightarrow&\sigma(a_{n+1})-a_{n+1}>\sigma(a_{n})-a_n.
\end{align*}

\end{proof}

\begin{prop}
Let
$$
g(n)=\frac{\sigma(n)^{\sigma(n)}}{n^n}.
$$
Then $g$ is increasing for $a_n\in SA$.
\end{prop}
\begin{proof}
Indeed, from definition of SA numbers and Lemma \ref{sigma n-n}
\begin{align*}
\frac{\sigma(a_{n+1})^{\sigma(a_{n+1})}}{(a_{n+1})^{a_{n+1}}}=&\left(\frac{\sigma(a_{n+1})}{a_{n+1}}\right)^{\sigma(a_{n+1})}(a_{n+1})^{\sigma(a_{n+1})-a_{n+1}}\\
                                 >&\left(\frac{\sigma(a_{n})}{a_n}\right)^{\sigma(a_{n+1})}(a_{n+1})^{\sigma(a_n)-a_n}\\
                                 >&\left(\frac{\sigma(a_{n})}{a_n}\right)^{\sigma(a_{n})}(a_{n})^{\sigma(a_n)-a_n}\\
                                 =&\frac{\sigma(a_{n})^{\sigma(a_{n})}}{(a_{n})^{a_{n}}}.
\end{align*}

\end{proof}

Now we give a stronger lemma.
\begin{thm}\label{sigma n-nloglogn increasing}
Let
$$
g(n)=\sigma(n)-n\log\log n.
$$
Then $g$ is increasing for large enough $n\in SA$.
\end{thm}
\begin{proof}
Let $n,\, n'$ be two consecutive superabundant numbers. By Lemma \ref{x<y<cx hx}, Proposition \ref{an+1<=2an} and inequality (\ref{x<y<cx loglog x}), with $c=2$, $x=n$, $y=n'$,  we have
\begin{align*}
\frac{1}{n'/n-1}\left(\frac{n'}{n}\frac{\log\log n'}{\log\log n}-1\right)\leq&\ 2\,\frac{\log\log (2n)}{\log\log n}-1\\
<&1+2\frac{\log 2}{\log n\log\log n}\\
<&\frac{\log\log12}{\log\log6},\qquad(n\geq24).
\end{align*}
This gives
\begin{equation}\label{loglog6/loglog12 2}
\frac{n'}{n}-1>\frac{\log\log6}{\log\log12}\left(\frac{n'}{n}\frac{\log\log n'}{\log\log n}-1\right).
\end{equation}
By definition of SA numbers
$$
\frac{\sigma(n')}{\sigma(n)}>\frac{n'}{n}.
$$
Hence, via (\ref{loglog6/loglog12 2}) we derive
\begin{align}\label{sigma-1>}
\frac{\sigma(n')}{\sigma(n)}-1>&\frac{n'}{n}-1\nonumber\\
=&\frac{\log\log6}{\log\log12}\frac{n'\log\log n'}{n\log\log n}-\frac{\log\log6}{\log\log12}\frac{n'\log\log n'}{n\log\log n}+\frac{n'}{n}-1\nonumber\\
>&\frac{\log\log6}{\log\log12}\left(\frac{n'\log\log n'}{n\log\log n}-1\right).\\\nonumber
\end{align}
On the other hand since $\frac{\log\log12}{\log\log6}<1.56077<e^\gamma$, by Proposition \ref{|f(n)-ey|<1/loglogn2}, for large enough $n$

\begin{equation}\label{sigma>loglog12 6}
\sigma(n)>\frac{\log\log12}{\log\log6}(n\log\log n).
\end{equation}
Multiplying both sides of (\ref{sigma-1>}) and (\ref{sigma>loglog12 6}), we have
$$
\sigma(n')-\sigma(n)>n'\log\log n'-n\log\log n.
$$
Therefore,
$$
\sigma(n')-n'\log\log n'>\sigma(n)-n\log\log n.
$$
\end{proof}

\begin{prop}
Let
$$
g(n)=\frac{\sigma(n)^{\sigma(n)}}{(n\log\log n)^{n\log\log n}}.
$$
Then $g$ is increasing for large enough $n\in SA$.
\end{prop}
\begin{proof}
By Proposition \ref{|f(n)-ey|<1/loglogn2} we have for large enough $n\in SA$

\begin{equation}\label{sigma an>3/2anloglog an}
\sigma(n)>\frac32n\log\log n.
\end{equation}
We show that for two consecutive superabundant $n,\, n'$
\begin{align*}
\frac{\sigma(n')^{\sigma(n')}}{(n'\log\log n')^{n'\log\log n'}}>\frac{\sigma(n)^{\sigma(n)}}{(n\log\log n)^{n\log\log n}}.
\end{align*}
Indeed,
\begin{align}
\frac{\sigma(n')^{\sigma(n')}}{\sigma(n)^{\sigma(n)}}\frac{(n\log\log n)^{n\log\log n}}{(n'\log\log n')^{n'\log\log n'}}=&\left(\frac{\sigma(n')}{\sigma(n)}\right)^{\sigma(n')}\left(\frac{n\log\log n}{n'\log\log n'}\right)^{n'\log\log n'}\label{a}\\
&\times\left\{\frac{\sigma(n)^{\sigma(n')-\sigma(n)}}{(n\log\log n)^{n'\log\log n'-n\log\log n}}\right\}.\nonumber
\end{align}
By Theorem \ref{sigma n-nloglogn increasing}, the term inside $\{\}$ is greater than 1. Moreover,\\
\begin{align*}
\left(\frac{\sigma(n')}{\sigma(n)}\right)^{\sigma(n')}\left(\frac{n\log\log n}{n'\log\log n'}\right)^{n'\log\log n'}>&\left(\frac{n'}{n}\right)^{\sigma(n')}\left(\frac{n\log\log n}{n'\log\log n'}\right)^{n'\log\log n'}\\
=&\left(\frac{n'}{n}\right)^{\sigma(n')-n'\log\log n'}\left(\frac{\log\log n}{\log\log n'}\right)^{n'\log\log n'}.
\end{align*}
However, due to (\ref{sigma an>3/2anloglog an}) the right-hand side of the equality is greater than

\begin{equation}\label{the right side of eq}
\left(\frac{n'}{n}\right)^{\frac12n'\log\log n'}\left(\frac{\log\log n}{\log\log n'}\right)^{n'\log\log n'}.
\end{equation}
Finally appealing to Lemma \ref{sqrt>loglog} we conclude that (\ref{the right side of eq}) is $>1$.
\end{proof}

\begin{prop}
Let $A=\{a_n\}$ be a sequence for which any prime factor of $a_n$ be a prime factor of $a_{n+1}$, and
$$
g(n)=\frac{n}{\varphi(n)},
$$
Then $g$ is increasing for $a_n\in A$.
\end{prop}
\begin{proof}
If $p(a_{n+1})=p(a_n)$, it is clear. Let $p(a_{n+1})=p_{k+1}>p_k= p(a_n)$
\begin{align*}
\frac{a_{n+1}/\varphi(a_{n+1})}{a_{n}/\varphi(a_{n})}&=\frac{\prod_{j=1}^{k}(1-1/p_j)}{\prod_{j=1}^{k+1}(1-1/p_j)}\\
                                                   &=\frac{1}{1-1/p_{k+1}}>1
\end{align*}
\end{proof}

\begin{prop}
Let $p(a_{n+1})\geq p(a_n)$,
$$
g(n)=\frac{\sigma(n)}{\varphi(n)},
$$
then $g$ is increasing for $a_n\in SA$.
\end{prop}
\begin{proof}
We have
\begin{align*}
\frac{\sigma(a_{n+1})}{\varphi(a_{n+1})}>\frac{a_{n+1}}{a_{n}}\frac{\sigma(a_{n})}{\varphi(a_{n+1})}=\frac{a_{n+1}}{\varphi(a_{n+1})}\frac{\sigma(a_{n})}{a_{n}}\geq\frac{a_{n}}{\varphi(a_{n})}\frac{\sigma(a_{n})}{a_{n}}=\frac{\sigma(a_{n})}{\varphi(a_{n})}.
\end{align*}
\end{proof}
Let $\Psi(n)$ denote Dedekind's arithmetical function of $n$ which is defined by
$$
\Psi(n)=n\prod_{p\mid n}\left(1+\frac1p\right),\qquad \Psi(1)=1,
$$
where the product is taken over all primes $p$ dividing $n$.
\begin{prop}
If $p(a_{n+1})\geq p(a_n)$, then
\begin{align*}
\frac{\sigma(n)^{\Psi(n)}}{n^n}.
\end{align*}
is increasing for $a_n\in SA$.
\end{prop}
\begin{proof}
Let $p(a_{n+1})\geq p(a_n)$. Then
$$
\Psi(a_{n+1})-a_{n+1}>\Psi(a_{n})-a_{n}.
$$
So that
\begin{align*}
\frac{\sigma(a_{n+1})^{\Psi(a_{n+1})}}{(a_{n+1})^{a_{n+1}}}=&\left(\frac{\sigma(a_{n+1})}{a_{n+1}}\right)^{\Psi(a_{n+1})}a_{n+1}^{\Psi(a_{n+1})-a_{n+1}}
                                                           >\left(\frac{\sigma(a_{n})}{a_{n}}\right)^{\Psi(a_{n+1})}a_{n+1}^{\Psi(a_{n+1})-a_{n+1}}\\
                                                           >&\left(\frac{\sigma(a_{n})}{a_{n}}\right)^{\Psi(a_{n})}a_{n+1}^{\Psi(a_{n+1})-a_{n+1}}
                                                           >\left(\frac{\sigma(a_{n})}{a_{n}}\right)^{\Psi(a_{n})}a_{n}^{\Psi(a_{n})-a_{n}}\\
                                                           =&\frac{\sigma(a_{n})^{\Psi(a_{n})}}{(a_{n})^{a_{n}}}.
\end{align*}

\end{proof}

\paragraph{\textbf{Colossally Abundant Numbers}}
A colossally abundant number is a positive integer $N$ for which there exists an $\varepsilon >0$ such that

\begin{equation}\label{CA defn}
\frac{\sigma(N)}{N^{1+\varepsilon}}\geq\frac{\sigma(n)}{n^{1+\varepsilon}},\qquad(n>1).
\end{equation}
It is easily seen that $CA\subset SA$.

For $\varepsilon>0$, we define $x=x_1$ and
\begin{align}\label{x(epsilon)}
F(x,1)=&\frac{\log(1+1/x)}{\log x}=\varepsilon,\\
F(x_k,k)=&\frac{\log(1+1/(x_k+\cdots+x_k^k))}{\log x}=\varepsilon\nonumber.
\end{align}
If $N=\prod_p p^{\nu_p(N)}$ is a CA number of parameter $\varepsilon$ and $p$ divides $N$ with $\nu_p(N) = k$, then applying (\ref{CA defn}) with $n = Np$ yields
$$
\varepsilon\geq F(p,k+1)\qquad\text{i.e. } p\geq x_{k+1}
$$
while, if $k >0$, applying (\ref{CA defn}) with $n = N/p$ yields
\begin{equation}\label{eps<F(p,k)}
\varepsilon\leq F(p,k),\qquad\text{i.e. }p\leq x_k
\end{equation}
Let $K$ be the largest integer such that $x_K\geq2$. Then from ((\ref{eps<F(p,k)}), for all $p$'s we have $2\leq p\leq x_k$ and
$$
k=\nu_p(N)\leq K
$$
Now define the set
$$
\mathcal{E}:=\{F(p,k): p \text{ is prime and } k\geq1\}.
$$
If $\varepsilon\notin\mathcal{E}$, then no $x_k$ is a prime and there exists a unique CA number $N=N(\varepsilon)$ of parameter $\varepsilon$; moreover, $N$ is given by either
\begin{equation}\label{colossally rep}
N_\varepsilon=\prod_{p=2}^x p^{k_p},\qquad(x_{k_p+1}<p<x_{k_p})
\end{equation}
or
\begin{equation}\label{CA N=prod p}
N=\prod_{k=1}^K\prod_{p<x_k}p.
\end{equation}
If $\varepsilon\in\mathcal{E}$, then some $x_k$ is prime, and it is highly probable that only one $x_k$ is prime. But from theorem of six exponentials it is only possible to show that at most two $x_k$'s are prime. Therefore, there are either two or four CA numbers of parameter $\varepsilon$, defined by
\begin{equation}\label{CA N=prod two or four}
N=\prod_{k=1}^K\prod_{\substack{p< x_k\\\text{or}\\p\leq x_k}}p.
\end{equation}
Here, if $x_k$ is a prime $p$ for some $k$, then $p$ may or may not be a factor in the inner product. (This can occur for at most two values of $k$.) In other words, if $x_{k-1} <p< x_k$, then the exponent $\nu_p(N)$ of $p$ in $N$ is $k$, while if $p = x_k$ , the exponent may be $k$ or $k-1$. In particular, if $N$ is the largest CA number of parameter $\varepsilon$, then
\begin{equation}\label{p(N)=p}
F(p,1)=\varepsilon\Rightarrow p(N)=p
\end{equation}
where $p(N)$ is the largest prime factor of $N$. Note that, since if $\varepsilon\notin\mathcal{E}$, then $x_k$ is not prime, formula (\ref{CA N=prod p}) gives the same value as (\ref{CA N=prod two or four}). Therefore, for any $\varepsilon$, formula (\ref{CA N=prod two or four}) gives all possible values of a CA number $N$ of parameter $\varepsilon$ (\cite{Caveney-2}).
For more details we refer to \cite{Al}, \cite{Erd}, \cite{Rob} and \cite{Caveney-2}.\\

It was proved by Robin (\cite{Rob}, Proposition 1) that the maximum order of the function $f$ defined by (\ref{f(n)}) is attained in a colossally abundant number.
\begin{xthm}[\cite{Rob}, Prop. 1]
If $3\leq N\leq n\leq N'$, where $N$ and $N'$ are two successive colossally abundant numbers, then
\begin{equation}\label{f(n)<max f(N), f(N')}
f(n)\leq\max\{f(N), f(N')\}.
\end{equation}

\end{xthm}
This fact shows, that if there is a counterexample to (\ref{Ro}), then there exists at list one colossally abundant number which violates it.

\begin{cor}\label{N<n<N' ca xa}
Let $N<N'$ be two consecutive CA numbers. If there exists an XA number $n>10080$ satisfying $N<n<N'$, then $N'$ is also an XA.
\end{cor}
\begin{proof}
Let $N<N'$ be two consecutive CA number. If $10080<n\in XA$ is such that $N<n<N'$, then put
$$
X=\{m\in XA:\ N<m<N'\}.
$$
By the assumption as $n\in XA$, we have $X\neq\emptyset$. Let $n'=\max X$. Since $n'\in XA$ and $n'>N$, then $f(n')>f(N)$. From inequality (\ref{f(n)<max f(N), f(N')}) we must have $f(n')<f(N')$. Hence, $N'\in XA$.
\end{proof}
\begin{rem}
In the case $N<n=10080<N'$, we have $N=5040$, $N'=55440$ and
$$
f(N)\approx1.790973367,\qquad f(n)\approx1.755814339,\qquad f(N')\approx1.751246515.
$$
Hence inequality (\ref{f(n)<max f(N), f(N')}) satisfies with $f(n)<f(N)=\max\{f(N),\ f(N')\}$.
\end{rem}

\begin{thm}
If RH holds, then there exist infinitely many CA numbers that are also XA.
\end{thm}
\begin{proof}
If RH holds, then by Theorem \ref{rh}, $\#XA=\infty$. Let $n$ be in XA. Since $\#CA=\infty$ (see \cite{Al}, \cite{Erd}), there exist two successive colossally abundant numbers $N,\ N'$ such that $N<n\leq N'$. If $N'=n$ then it is readily in XA, otherwise $N'$ belongs to XA via Corollary \ref{N<n<N' ca xa}.
\end{proof}
It can be seen that there exist infinitely many CA numbers $N$ for which the largest prime factor $p$ is greater than $\log N$. For this purpose, we use the following lemma

\begin{lem}[\cite{Caveney-2}, Lemma 3]\label{log N<theta+cx1/2}
Let $N$ be a CA number of parameter $\varepsilon<F(2, 1) = \log(3/2)/ \log 2$ and define $x=x(\varepsilon)$ by (\ref{x(epsilon)}). Then
\begin{enumerate}[\upshape (i)]
\item for some constant $c >0$
$$
\log N\leq\vartheta(x)+c\sqrt{x}.
$$
\item Moreover, if $N$ is the largest CA number of parameter $\varepsilon$, then
$$
\vartheta(x)\leq\log N\leq\vartheta(x)+c\sqrt{x}.
$$
\end{enumerate}
\end{lem}
\begin{lem}[\cite{Caveney-2}, Lemma 4]\label{lemma 4, caveney 2}
There exists a constant $c > 0$ such that for infinitely many primes $p$ we have
\begin{equation}\label{theta p-p<cp1/2logloglog p}
\vartheta(p)<p-c\sqrt{p}\log\log\log p,
\end{equation}
and for infinitely many other primes $p$ we have
$$
\vartheta(p)>p+c\sqrt{p}\log\log\log p.
$$
\end{lem}

\begin{thm}\label{log N<p(N)}
There are infinitely many CA numbers $N_\varepsilon$, such that $\log N_\varepsilon<p(N_\varepsilon)$.
\end{thm}
\begin{proof}
Choose $p$ as in (\ref{theta p-p<cp1/2logloglog p}) and $N_\varepsilon$ the largest CA-number of parameter
$$
\varepsilon=F(p,1).
$$
Then, from (\ref{p(N)=p}), one has $p(N_\varepsilon) = p$. By Lemma \ref{log N<theta+cx1/2}(ii)
\begin{align*}
\log N_\varepsilon-\vartheta(p)<c\sqrt{p},\qquad(\text{for some}\ c>0).
\end{align*}
On the other hand, by Lemma \ref{lemma 4, caveney 2}, There exists a constant $c' > 0$ such that for infinitely
many primes $p$ we have
$$
\vartheta(p)-p<-c'\sqrt{p}\log\log\log p,\qquad(c'>0).
$$
Hence, for any $x_0$, there exists an $x>x_0$ such that
$$
\log N_\varepsilon-p<\{c-c'\log\log\log p\}\sqrt{p}<0.
$$
We get the desired result.
\end{proof}

\paragraph{\textbf{Extremely Abundant Numbers}} Returning to extremely abundants, we present some properties of them.
\begin{thm}\label{p<log n prop}
Let $n=2^{k_2}\cdots p$ be an extremely abundant number. Then
\[
p<\log n.
\]
\end{thm}
\begin{proof}
For $n=10080$ we have
$$
p(10080)=7<9.218<\log(10080).
$$
Let $n>10080$ be an extremely abundant number and $m=n/p$. Then $m>10080$, since the only superabundant numbers between $10080$ and $11\times10080$ are $\{$10080,\ 15120,\ 25200,\ 27720,\ 55440,\ 110880$\}$ and computation shows that non of them are in XA. Hence by definition
\[
\frac{\sigma(n)/n}{\sigma(m)/m}>\frac{\log\log n}{\log\log m}.
\]
So
\[
1+\frac{1}{p}>\frac{\log\log n}{\log\log m}\Rightarrow \frac{1}{p}>\frac{\log(1+\log p/\log m)}{\log\log m}.
\]\\
Using inequality (\ref{log(1+t)}) we have
\[
\frac{1}{p}>\frac{\log p}{\log n\log\log m}>\frac{\log p}{\log n\log\log n}\Rightarrow p<\log n.
\]

\end{proof}
We mention a similar result proved by Choie et al.
\begin{prop}[\cite{Choie}, Lemma 6.1]
Let $t\geq2$ be fixed. Suppose that there exists a $t$-free integer exceeding 5040 that does not satisfy Robin's inequality. Let $n=2^{k_2}\cdots p$ be the smallest such integer. Then $p<\log n$.
\end{prop}

In the previous section we showed that, if RH holds, then there exist infinitely many CA numbers that are also XA. Next theorem is a conclusion of Theorems \ref{log N<p(N)} and \ref{p<log n prop} which is independent of RH.
\begin{thm}
There exist infinitely many CA numbers that are not XA.
\end{thm}

\begin{thm}\label{sigma(n)+phi(n)/n}
Let
\begin{align*}
g(n)=\frac{\sigma(n)+\varphi(n)}{n}.
\end{align*}
For two consecutive extremely abundant numbers $n=2^{k}\cdots p$ and $n'=2^{k'}\cdots p'$, if $p'\geq p$ and $\log (n'/n)>1/\log p$, then $g(n)<g(n')$ for large enough $n,\ n'\in XA$.
\end{thm}
\begin{proof}
If the largest primes of $n$ and $n'$ are equal, it is clear. Let $p'=p_{k+1}>p_k=p$.
If $n>10080$ is extremely abundant, then
$$
\frac{\sigma(n)}{n\log\log n}>\frac{\sigma(10080)}{10080\log\log 10080}>1.75.
$$
Using inequality (\ref{log(1+t)}), Proposition \ref{p(1+e)<log an<p(1+e)}, Lemma \ref{prod1/(1-1/p)>e...} and Lemma \ref{pk+1<pk(1+1/log2 pk)}, we deduce for large enough $n$

\begin{align*}
\frac{\sigma(n')}{n'}+&\frac{\varphi(n')}{n'}-\frac{\sigma(n)}{n}-\frac{\varphi(n)}{n}\\
                           >&\frac{\sigma(n)}{n}\frac{\log\log n'-\log\log n}{\log\log n}-\frac{1}{p_{k+1}}\prod_{j=1}^k\left(1-\frac{1}{p_j}\right)\\
                           >&1.75\log\frac{\log n'}{\log n}-\frac{1}{p_{k+1}}\prod_{j=1}^k\left(1-\frac{1}{p_j}\right)\\
                           >&1.75\frac{\log (n'/n)}{\log n'}-\frac{1}{p_{k+1}}\prod_{j=1}^k\left(1-\frac{1}{p_j}\right)\\
                           >&1.75\frac{\log (n'/n)}{\log n'}-\frac{1}{p_{k+1}}\frac{e^{-\gamma}}{\log p_k}\left(1+\frac{0.2}{\log^2 p_k}\right)\\
                           >&1.75\frac{\log (n'/n)}{p_{k+1}(1+\frac{0.5}{\log p_{k+1}})}
                           -\frac{1}{p_{k+1}}\frac{e^{-\gamma}}{\log p_k}\left(1+\frac{0.2}{\log^2 p_k}\right)\\
                           >&\frac{1}{p_{k+1}\log p_k}\left\{\frac{1.75}{(1+\frac{0.5}{\log p_{k+1}})}
                           -e^{-\gamma}\left(1+\frac{0.2}{\log^2 p_k}\right)\right\}\\
                           >&0.
\end{align*}

\end{proof}

\begin{rem}
We checked that the assumptions in Theorem \ref{sigma(n)+phi(n)/n} hold up to 8150-th element of XA.
\end{rem}

\section{Numerical experiments}
In this section we give some numerical results for the set of extremely abundant numbers up to its 13770-th element, which is less than $C_1=s_{500,000}$ (i.e. $500,000$-th superabundant number) basing on the list provided by T. D. Noe \cite{Noe}. We examined Property 1 to 4 and Remark 1 below for the corresponding extremely abundant numbers extracted from the list.
\begin{property}
Let $n=2^{k_2}\cdots q^{k_q}\cdots r^{k_r}\cdots p$ be an extremely abundant number, where $2\leq q<r\leq p$. Then for $10080<n\leq C_1$
\begin{enumerate}[\upshape (i)]
  \item $\log n<q^{k_q+1},$
  \item $r^{k_r}<q^{k_q+1}<r^{k_r+2},$
  \item $q^{k_q}<k_qp,$
  \item $q^{k_q}\log q<\log n\log\log n<q^{k_q+2}.$
  \end{enumerate}
\end{property}
\begin{property}
Let $n=2^{k_2}\cdots x_k^k.\cdots p$ be an extremely abundant number, where $2<x_k<p$ is the greatest prime factor of exponent k. Then
$$
\sqrt{p}<x_2<\sqrt{2p},\qquad\text{for}\ \ 10080<n\leq C_1.
$$
\end{property}
\begin{property}
Let $n=2^{\alpha_2}\cdots q^{\alpha_q}\cdots p$ and $n'=2^{\beta_2}\cdots q^{\beta_q}\cdots p'$  be two consecutive extremely abundant numbers greater than $10080$. Then for $10080<n\leq C_1$\\[-3pt]
$$
\alpha_q-\beta_q\in\{-1, 0, 1\},\qquad \mbox{for all}\ \ 2\leq q\leq p'.
$$
\end{property}
\begin{property}\label{d-p}
If $m, n$ are extremely abundant and $m<n$, then for $10080<n\leq C_1$\\[-3pt]
\begin{enumerate}[\upshape (i)]
  \item $p(m)\leq p(n)$,
  \item $d(m)\leq d(n)$.
\end{enumerate}

\end{property}
\begin{rem}
Note that the latter Property does not imply that an extremely abundant number to be a highly composite one. For example,
\[
n_1=(139\sharp)(13\sharp)(5\sharp)(3\sharp)^2(2)^4,
\]
\[
n_2=(149\sharp)(13\sharp)(7\sharp)(5\sharp)(3\sharp)(2)^5,
\]
\[
n_3=(151\sharp)(13\sharp)(5\sharp)(3\sharp)^2(2)^3.
\]
$n_1, n_3$ are consecutive extremely abundant numbers and $n_3>n_2>n_1$;
\[
f(n_3)>f(n_2)\ ,\qquad \mbox{but}\qquad d(n_3)<d(n_2).
\]
\end{rem}
\begin{rem}
We note that Property \ref{d-p} is not true for superabundant numbers. For example
$$
s_{47}=(19\sharp)(3\sharp)^22,\qquad s_{48}=(17\sharp)(5\sharp)(3\sharp)2^3,\qquad p(s_{48})=17<19=p(s_{47}).
$$
and
$$
s_{173}=(59\sharp)(7\sharp)(5\sharp)(3\sharp)^22^3,\qquad s_{174}=(61\sharp)(7\sharp)(3\sharp)^22^2,\qquad \frac{d(s_{174})}{d(s_{173})}=\frac{35}{36}<1,
$$
where $s_k$ denotes $k$-th superabundant number.
\end{rem}
Using Table of superabundant and colossally abundant numbers in \cite{Noe} we have
\begin{align*}
\#&\{n\in XA: n<C\}=24,875,\\
\#&\{n\in CA: n<C\}=21,187,\\
\#&\{n\in CA\cap XA: n<C\}=20,468,\\
\#&\{n\in CA\setminus XA: n<C\}=719,\\
\#&\{n\in XA\setminus CA: n<C\}=4407,
\end{align*}
where $C=s_{1000,000}$.

\newpage
\begin{table}
\begin{tabular}{||c|c|c|c|c|c|c||}
\hline
  & $n$          &Type& $f(n)$    & $p(n)$ & $\log n$ & $k_2$\\[3pt]
  \hline
  1&$\ \quad(7\sharp)(3\sharp)2^3=10080$                                      &s & $1.75581$ & $7$ & $9.21831$ & $5$\\[3pt]
  2&$(113\sharp)(13\sharp)(5\sharp)(3\sharp)^2 2^3$&c & $1.75718$ & $113$ & $126.444$ & $8$\\[3pt]
  3&$(127\sharp)(13\sharp)(5\sharp)(3\sharp)^2 2^3$&c & $1.75737$ & $127$ & $131.288$ & $8$\\[3pt]
  4&$(131\sharp)(13\sharp)(5\sharp)(3\sharp)^2 2^3$&c & $1.75764$ & $131$ & $136.163$ & $8$\\[3pt]
  5&$(137\sharp)(13\sharp)(5\sharp)(3\sharp)^2 2^3$&c & $1.75778$ & $137$ & $141.083$ & $8$\\[3pt]
  6&$(139\sharp)(13\sharp)(5\sharp)(3\sharp)^2 2^3$&c & $1.75821$ & $139$ & $146.018$ & $8$\\[3pt]
  7&$(139\sharp)(13\sharp)(5\sharp)(3\sharp)^2 2^4$&c & $1.75826$ & $139$ & $146.711$ & $9$\\[3pt]
  8&$(151\sharp)(13\sharp)(5\sharp)(3\sharp)^2 2^3$&s & $1.75831$ & $151$ & $156.039$ & $8$\\[3pt]
  9&$(151\sharp)(13\sharp)(5\sharp)(3\sharp)^2 2^4$&c & $1.75849$ & $151$ & $156.732$ & $9$\\[3pt]
  10&$(151\sharp)(13\sharp)(7\sharp)(3\sharp)^2 2^4$&c & $1.75860$ & $151$ & $158.678$ & $9$\\[3pt]
  \hline
\end{tabular}
\caption{First 10 extremely abundant numbers ( s is superabundant, c is colossally abundant and $p_k\sharp=\prod_{j=1}^k p_j$).}\label{tabl}
\end{table}
\begin{rem}
Here the  so-called primorial of a prime $p$ is denoted by $p\sharp$ and this is an analog of the usual factorial for primes. The name was suggested by Dubner (\cite{Rib}, p. 12) and it is
\[
p_k\sharp=\prod_{i=1}^k p_i.
\]
\end{rem}
The following properties has been checked up to $C_2$ (250,000-th element of SA numbers) and for 8150-th element of XA numbers in this domain.
\begin{property}
If $n,\ n'\in SA$ are consecutive, then
$$
\frac{\sigma(n')/n'}{\sigma(n)/n}<1+\frac{1}{p\,'},\qquad(n'<C_2).
$$
\end{property}

\begin{property}\label{an+1/an>1+cloglog an 2/log an}
If $n,\ n'\in XA$ are consecutive, then
$$
\frac{n'}{n}>1+c\,\frac{(\log\log n)^2}{\log n},\qquad(0<c\leq4,\ n'<C_2),
$$
$$
\frac{n'}{n}>1+c\,\frac{(\log\log n)^2}{\sqrt{\log n}},\qquad(0<c\leq0.195,\ n'<C_2).
$$
\end{property}

\begin{property}
If $n, n'\in XA$ are consecutive, then
$$
\frac{f(n')}{f(n)}<1+\frac{1}{p\,'},\qquad (n'<C_2).
$$
\end{property}
The number of distinct prime factors of a number $n$ is denoted by $\omega(n)$ (\cite{Sandor}). From Property \ref{an+1/an>1+cloglog an 2/log an} we easily can get
$$
g(n)=\frac{n}{\omega(n)}
$$
is increasing for $n\in XA$, where $n<C_2$.\\
\begin{property}
The composition
\begin{align*}
\sigma\left(n\left\lfloor\frac{\sigma(n)}{n}\right\rfloor\right)
\end{align*}
is increasing for $n\in SA$, $n<C_2$.
\end{property}

\begin{property}
Let $g$ be
\begin{align*}
(1)&\ \frac{\sigma(n)^{\varphi(n)}}{n^n}&(2)&\ \frac{\Psi(n)^{\varphi(n)}}{n^n}.
\end{align*}
Then, $g$ is decreasing for $n\in SA$, $n<C_2$.
\end{property}

\begin{property}
Let $g$ be
\begin{align*}
(1)&\ \frac{\Psi(n)^{\sigma(n)}}{n^n}&(2)&\ \frac{\varphi(n)^{\sigma(n)}}{n^n},\qquad(a_n>a_3)\\
(3)&\ \frac{\varphi(n)^{\Psi(n)}}{n^n},\qquad(a_n>a_3,\ p(a_{n+1})\geq p(a_n)).
\end{align*}
Then, $g$ is increasing for $n\in SA$, $n<C_2$.
\end{property}

\begin{property}
Let $g$ be each of the following arithmetic functions:
\begin{align*}
(1)&\ \frac{\varphi(n)}{\varphi(\varphi(n))}&(2)&\ \frac{n}{\varphi(\varphi(n))}\\
(3)&\ d(n)\omega(n)&(4)&\ \omega(\varphi(n)).
\end{align*}
Then $g$ is increasing for $n\in XA$, $n<C_2$.
\end{property}

\begin{property}
The compositions
\begin{align*}
(1)&\ \varphi\left(n\left\lfloor\frac{\sigma(n)}{n}\right\rfloor\right)&(2)&\ \varphi\left(n\left\lfloor\frac{n}{\varphi(n)}\right\rfloor\right)
&(3)&\ \varphi\left(n\left\lfloor\frac{\Psi(n)}{n}\right\rfloor\right)
\end{align*}
are increasing for $n\in XA$, $n<C_2$.
\end{property}

\begin{property}
Let $g(m)=\lcm(1,2,\ldots,m)$. Let $n=2^{k_2}\cdots p\in SA$, then
$$
f(n)>f(g(p)),\qquad(s_{49}<n<C_2)
$$
\end{property}


\section*{Acknowledgments}
The work of the second author is supported by Center of Mathematics of the University of Porto.
We thank J. C. Lagarias, C. Calderon, J. Stopple and M. Wolf for useful
discussions and sending us some relevant references. Our sincere thanks to J. L.
Nicolas for careful reading of the manuscript, helpful comments and worthwhile
suggestions which rather improved the presentation of the paper.
The work of the first author is supported by the Calouste Gulbenkian Foundation, under
Ph.D. grant number CB/C02/2009/32. Research partially funded by the European Regional
Development Fund through the programme COMPETE and by the Portuguese Government
through the FCT under the project PEst-C/MAT/UI0144/2011. Sincere thanks to \'{E}lio Coutinho from Informatics Center of Faculty of Science of the University of Porto for providing accesses to fast computers for necessary computations.

\bibliographystyle{plain}
\bibliography{bibliojab1}

\vspace*{10mm}
\noindent Email: sdnazdi@yahoo.com\\[10pt]
Email: syakubov@fc.up.pt\\
\end{document}